\documentclass[preprint]{elsarticle}

\usepackage{amsmath}
\usepackage{amsfonts}
\usepackage{amssymb}
\usepackage{amsthm}

\newtheorem{remark}{\it Remark}[section]

\newtheorem{definition}{\it Definition}[section]
\newtheorem{proposition}{\it Proposition}[section]
\newtheorem{theorem}{\it Theorem}[section]
\newtheorem{lemma}{\it Lemma}[section]

\newcommand{\tra}[1]{\,{\vphantom{#1}}^{\textsc{t}}\!{#1}}

\begin{document}

\title{Discrete calculus of variations for quadratic lagrangians}

\author[LMPA]{P.~Ryckelynck\corref{cor1}}
\ead{ryckelyn@lmpa.univ-littoral.fr}
\author[LMPA]{L.~Smoch}
\ead{smoch@lmpa.univ-littoral.fr}

\address[LMPA]{ULCO, LMPA, F-62100 Calais, France\\Univ Lille Nord de France, F-59000 Lille, France. CNRS, FR 2956, France.}

\begin{abstract}
We develop in this paper a new framework for discrete calculus of variations when the actions have densities involving an arbitrary discretization operator. We deduce the discrete Euler-Lagrange equations for piecewise continuous critical points of sampled actions. Then we characterize the discretization operators such that, for all quadratic lagrangian, the discrete Euler-Lagrange equations converge to the classical ones. 
\end{abstract}


\maketitle

\section{Introduction}

Discrete variational problems and discrete mechanics constitute active fields of research and have been thoroughly studied by many authors
among which we may cite J. Marsden and J. Cresson, see for instance \cite{Cre,CFT,MW,OJM,SLO}. A first approach has its roots in the work of J. A. Cadzow, developed in the 1970s, and consists in replacing dynamical variables and paths by discrete variables and paths which are finite linear combinations of indicator functions. Based on this discretization, the action integral, the lagrangian and the energy are approximated on a time slice $[k\varepsilon,(k+1)\varepsilon]$, for some time delay $\varepsilon$, by discrete analogs which are finite sums depending of a finite number of variables in $\mathbb{R}^d$.  The presentation of this theory, together with an history and large bibliography, is developed in \cite{OJM}. 
Another different approach consists in replacing the derivative $\mathbf{\dot x}(t)$ of the dynamical variable $\mathbf{x}(t)$ with a three terms scale derivative $\Box_{\varepsilon,\mathcal{Q}} \mathbf{x}(t)$. In this way, a dynamical variable $\mathbf{x}(t)$ remains even after the sampling process a function $\Box_{\varepsilon,\mathcal{Q}}\mathbf{x}(t)$ of ``continuous'' time. The principle of least action may be extended to the case of non-differentiable dynamical variables of H\"{o}lderian regularity $\mathcal{C}^{1/2}$. We refer to \cite{Cre,CFT} for further details among which the extension of Noether's Theorem and also an interesting informal discussion about fractal physics. While the first approach is widely used in numerical analysis of variational dynamical systems, the second one is, roughly speaking, more concerned with foundational discussion of microphysics.

In this paper we highlight and use the equations of motion 
\begin{equation}
\nabla _\mathbf{x} \displaystyle \mathcal{L} +
\Box_{-\varepsilon }\nabla_{\mathbf{\dot x}}\mathcal{L}=0,  
\label{intrinsic}
\end{equation}
for all lagrangian $\mathcal{L}$ and general discretization operator $\Box_\varepsilon$ defined as 
\begin{equation}
\Box _\varepsilon \mathbf{x}(t)=\sum_{\ell=-N}^Nc_\ell\mathbf{x}(t+\ell\varepsilon)\chi_{-\ell}(t).
\label{ourbox}
\end{equation}
In this formula, the symbols $\chi_{-\ell}(t)$ denote characteristic functions of some intervals, and prevent $t+\ell\varepsilon$ from belonging to an interval in which $\mathbf{x}$ is undefined.

A first motivation in this paper is to justify the choice of $\Box_{\varepsilon,\mathcal{Q}}$ in \cite{Cre,CFT}. To do this, we use throughout the previous generalization (\ref{ourbox}) and distinguish among the operators $\Box_\varepsilon$ those having some specific features. A second intend is to understand the similarities and differences between the equations of motions for discretized dynamical variables and the classical Euler-Lagrange equations $\mathcal{L}_\mathbf{x}'-d/dt\mathcal{L}'_{\dot{\mathbf{x}}}=0$. A third one is to generalize some results of \cite{Cre,CFT} to operators $\Box_\varepsilon$ not satisfying some Leibniz formula and to extremal curves $\mathbf{x}(t)$ which are not $C^{1/2}$ but rather of a weaker specific regularity $\mathcal{C}_{pw}$.

In Section 2, we first fix the notation used in the paper and study two specific sets $\mathcal{O}_{N,\varepsilon}$ and $\tilde{\mathcal{O}}_{N,\varepsilon}$ of operators (\ref{ourbox}). In Section 3, we describe a class of operators $\Box_\varepsilon$ satisfying an appropriate extension of Leibniz formula. 
In Section 4, we get the necessary first order condition for finding a minimizer of a discrete action $\mathcal{A}_{disc}(\mathbf{x)}$. We prove the equation (\ref{intrinsic}), which depends only on $\Box_\varepsilon $ and is reminiscent of Cresson's result. 
In Section 5, we first introduce the classical and discrete models for the quadratic lagrangians and emphasize on their similarity. Next, the subset of operators $\Box_\varepsilon$ for which discrete Euler Lagrange equations have oscillatory solutions is examined. In Section 6, we state the main result of the paper which characterizes the set of operators $\Box_\varepsilon$ for which the convergence of (\ref{intrinsic}) holds, for every quadratic lagrangian.

In this paper, C.E.L./D.E.L. stand for classical/discrete Euler-Lagrange equations.

\section{Notation and framework}

First, let us fix some notation. Let $[a,b]$ be some interval of time and a time delay $\varepsilon >0$ be fixed throughout. We denote by $d$ the ``physical" dimension and by $N$ the number of samples in $\mathbb{C}^d$. We use the notation $i$ for $\sqrt{-1}$. We denote by $\chi _\ell(t)$ the characteristic function of the interval $[\max(a,a+\ell\varepsilon ),\min (b,b+\ell\varepsilon )]$, for all integer $\ell$.

Many interesting ways to generate discrete actions arise in the following way. The derivative $\mathbf{\dot{x}}(t)$ occuring in lagrangians is replaced with a linear combination of discretized values of $\mathbf{x}(t)$
\begin{equation}
\Box^{[r,s]}_\varepsilon \mathbf{x}(t)=-\chi_{+1}(t)\frac{s}{\varepsilon}\mathbf{x}(t-\varepsilon)+\frac{s-r}{\varepsilon}\mathbf{x}(t)+\chi_{-1}(t)\frac{r}{\varepsilon}\mathbf{x}(t+\varepsilon)
\label{ourboxpq}
\end{equation}
where $r,s\in\mathbb{C}^\star$. In \cite{Cre,CFT}, the authors choose the values $r=(1-i)/2$ and $s=(1+i)/2$ but do not work with characteristic functions. Let $\Box_{\varepsilon,\mathcal{Q}}$ be defined up to now by $\Box_\varepsilon^{[\frac{1-i}{2}\frac{1+i}{2}]}$. Obviously, (\ref{ourboxpq}) is a particular case of (\ref{ourbox}) with $N=1$, $s=-c_{-1}\varepsilon$, $r=c_1\varepsilon$ and $c_0=-(c_{-1}+c_1)$. We generalize this process by using linear combinations of $2N+1$ terms of the shape $c_\ell\mathbf{x}(t+\ell\varepsilon )\chi_{-\ell}(t)$ where $c_\ell\in\mathbb{C}$ are fixed.

\begin{definition}
Given $\alpha$ and $\beta$ in $\mathbb{C}^d$, the affine space $\mathcal{C}_{pw}(d,N,\alpha,\beta)$ is the set of the functions $\mathbf{x}%
:[a,b]\rightarrow \mathbb{C}^d$ satisfying $\mathbf{x}(a)=
\mathbf{\alpha }$, $\mathbf{x}(b)=\mathbf{\beta }$ and continuous on each interval $[a+\ell\varepsilon
,a+(\ell+1)\varepsilon ]\cap [a,b]$ for all integer $\ell$.
\end{definition}

This space is endowed with the topology of uniform convergence, its tangent space is everywhere the Banach space $\mathcal{C}_{pw}(d,N,0,0)$.

\begin{definition}
For any $\mathbf{x}\in \mathcal{C}_{pw}(d,N,\alpha ,\beta )$, we
denote by $S(\mathbf{x})$ the row vector valued function 
\begin{equation}
S(\mathbf{x})(t)=(x_1(t-N\varepsilon )\chi _N(t)\ldots x_j(t+\ell\varepsilon
)\chi _{-\ell}(t)\ldots x_d(t+N\varepsilon )\chi _{-N}(t)).
\label{Sx}
\end{equation}
\end{definition}

In $S(\mathbf{x})$ the ordering of variables is lexicographic, first on $\ell$ and next on $j$. The function $S(\mathbf{x})$ lies in a product of $d(2N+1)$
affine spaces modeled on $\mathcal{C}_{pw}(1,N,0,0)$ and, in this way, $S$ is an injective continuous linear mapping. 

\begin{definition}
If $(c_\ell)\in \mathbb{C}^{2N+1}$, the generalized scale derivative is the continuous linear endomorphism of $\mathcal{C}_{pw}(d,N,\alpha ,\beta )$ defined by (\ref{ourbox}) for all $t\in [a,b]$. 
\end{definition}

If $d=1$, for all $\mathbf{x}\in \mathcal{C}_{pw}(1,N,\alpha,\beta)$, we have
\begin{center}
$\Box _\varepsilon \mathbf{x}=(c_{-N}\ldots c_0\ldots c_N)\tra(S(\mathbf{x}))$.
\end{center}
When $d\geq 2$, we have a similar relationship between $\Box _\varepsilon \mathbf{x}$ and $S(\mathbf{x})$, involving a $d\times d(2N+1)$ banded matrix 
\[
\Box _\varepsilon \mathbf{x}=\begin{pmatrix}
c_{-N}\ldots c_0\ldots c_N & & & \\
 & c_{-N}\ldots c_0\ldots c_N & & 0\\
0 & & \ddots & \\
 & & & c_{-N}\ldots c_0\ldots c_N
\end{pmatrix}\tra (S(\mathbf{x})). 
\]
Note that $\Box_\varepsilon$ is a well defined continuous endomorphism of the Banach space $\mathcal{C}_{pw}(d,N,\alpha,\beta)$.\\
The convergence of $\Box_\varepsilon \mathbf{x}$ to the ordinary derivative $\dot{\mathbf{x}}$ is connected to the following definitions.

\begin{definition}
The vector space $\mathcal{O}_{N,\varepsilon }$ is the set of operators $\Box
_\varepsilon $ of the shape (\ref{ourbox}) with coefficients $c_\ell=\gamma _\ell/\varepsilon $ and where $\gamma_\ell\in\mathbb{C}$ do not depend on $\varepsilon$. Algebraically, $\mathcal{O}_{N,\varepsilon }\simeq \mathbb{C}^{2N+1}$.
\end{definition}

\begin{definition}
The affine space $\tilde{\mathcal{O}}_{N,\varepsilon}\subset \mathcal{O}_{N,\varepsilon }$ contains operators $\Box_\varepsilon$ for which  $\Box_\varepsilon 1=0$ and $\Box_\varepsilon t=1$ when $t$ lies in the safety interval $\mathcal{I}_S=[a+2N\varepsilon
,b-2N\varepsilon ]$.
\label{Otilde}
\end{definition}

We shall see, as a consequence of Theorem \ref{Theorem-conv-schemas}, that if $\Box_\varepsilon\in\tilde{\mathcal{O}}_{N,\varepsilon}$, then $\Box_\varepsilon\mathbf{x}(t)$ tends to $\dot{\mathbf{x}}(t)$ locally uniformly in $]a,b[$, for all $\mathbf{x}\in\mathcal{C}^2([a,b],\mathbb{C}^d)$.

\begin{remark}\rm
The discrete Euler forward and backward difference operators are respectively defined by $\Delta_\varepsilon^+=\Box_\varepsilon^{[1,0]}$ and $\Delta_\varepsilon^-=\Box_\varepsilon^{[0,1]}$. Their mean is the symmetric difference operator $\displaystyle \Box_{\varepsilon,\mathcal{S}}=\Box_\varepsilon^{[\frac{1}{2},\frac{1}{2}]}$. These three operators and $\Box_{\varepsilon,\mathcal{Q}}$ are elements of $\tilde{\mathcal{O}}_{1,\varepsilon }$. Notice that the operator $\Box_{\varepsilon,\mathcal{Q}}$ is also related to the one-dimensional version of the operator used by Kime \cite{Kime} when she solves numerically the Schroedinger equation via Forward-Difference and Backward-Difference approximations at various steps in time.
\end{remark}

\begin{remark}\rm
The operator $\displaystyle\mathbf{x}\rightarrow\Delta_\varepsilon^+(\Delta_\varepsilon^+\mathbf{x}(t-\varepsilon))=\frac{1}{\varepsilon}\Box_\varepsilon^{[1,-1]}\mathbf{x}(t)$ approximates well the second derivative but does not lie in $\mathcal{O}_{1,\varepsilon}$.
\end{remark}

\section{Leibniz formulas for $\Box_\varepsilon$ operators in $\mathcal{O}_{1,\varepsilon}$}

In order to deduce his version of D.E.L. with $\Box_{\varepsilon,\mathcal{Q}}$, Cresson in \cite{Cre} found an analog of the classical Leibniz formula. When such a formula exists, a principle of discrete virtual works may be stated. In this section, we generalize Cresson's identity to the family of operators (\ref{ourboxpq}).

\begin{theorem}
Let $\Box_{\varepsilon}\in\mathcal{O}_{1,\varepsilon}$ of the shape (\ref{ourboxpq}) with $r,s\in\mathbb{C}^\star$ and $s/r\notin \mathbb{R}$ then, for all piecewise continuous functions $\mathbf{f},\mathbf{g}:[a,b]\rightarrow\mathbb{C}^d$, we get the generalized Leibniz formula
\begin{equation*}
\Box^{[r,s]}_\varepsilon (\mathbf{f}\cdot\mathbf{g})(t)=\mathbf{f}(t)\cdot\Box^{[r,s]}_\varepsilon \mathbf{g}(t)+
\mathbf{g}(t)\cdot\Box^{[r,s]}_\varepsilon \mathbf{f}(t)+
\end{equation*}
\begin{equation*}
\frac{\varepsilon(r\overline{s}^2-\overline{r}^2s)}{(r\overline{s}-\overline{r}s)^2}\Box^{[r,s]}_\varepsilon\mathbf{f}(t)\cdot\Box^{[r,s]}_\varepsilon \mathbf{g}(t)-\frac{\varepsilon rs(r-s)}{(r\overline{s}-\overline{r}s)^2}\Box^{[\overline{r},\overline{s}]}_\varepsilon \mathbf{f}(t)\cdot\Box^{[\overline{r},\overline{s}]}_\varepsilon \mathbf{g}(t)+
\end{equation*}
\begin{equation}
\frac{\varepsilon rs(\overline{r}-\overline{s})}{(r\overline{s}-\overline{r}s)^2}\left(\Box^{[r,s]}_\varepsilon \mathbf{f}(t)\cdot\Box^{[\overline{r},\overline{s}]}_\varepsilon \mathbf{g}(t)+\Box^{[\overline{r},\overline{s}]}_\varepsilon \mathbf{f}(t)\cdot\Box^{[r,s]}_\varepsilon \mathbf{g}(t)\right).
\label{Leibniz}
\end{equation} 
\label{thm1}
\end{theorem}
\begin{proof}
The formula (\ref{Leibniz}) is $\mathbb{C}$-bilinear w.r.t. $\mathbf{f}$ and $\mathbf{g}$. Having in mind the properties of the inner product and the fact that $\Box_\varepsilon$ acts component-wise, we may suppose without loss of generality that $d=1$ and $\mathbf{f},\mathbf{g}\in\mathbb{R}$.   
We slightly generalize the proof of Theorem 2.1 of \cite{Cre} which attempts to give a formula such as
\begin{center}
$W(fg)=W(f)g+fW(g)+$
\end{center}
\vspace{-0.2cm}
\begin{equation}
d_1W(f)W(g)+d_2\tilde W(f)W(g)+d_3W(f)\tilde W(g)+d_4\tilde W(f)\tilde W(g).
\label{Leib1}
\end{equation}
Here $W$ and $\tilde{W}$ are two operators in $\tilde{O}_{1,\varepsilon}$, $f$ and $g$ are arbitrary in $\mathcal{C}_{pw}(1,1,\alpha,\beta)$ and  $d_1,d_2,d_3,d_4$ are four complex numbers. Now, we choose $W=\Box_\varepsilon^{[r,s]}$ and $\tilde{W}=\Box_\varepsilon^{[r',s']}$ for some complex numbers $r,s,r',s'$ such that $rs'-sr'\neq 0$. We have obviously $W=r\Delta_\varepsilon^++s\Delta_\varepsilon^-$ and $\tilde W=r'\Delta_\varepsilon^++s'\Delta_\varepsilon^-$. But we have also the well-known formulas for $\Delta_\varepsilon^+$ and $\Delta_\varepsilon^-$
\begin{equation}
\Delta_\varepsilon^\sigma(fg)=\Delta_\varepsilon^\sigma(f)g+f\Delta_\varepsilon^\sigma(g)+\sigma\varepsilon\Delta_\varepsilon^\sigma(f)\Delta_\varepsilon^\sigma(g),
\end{equation}
where $\sigma=\pm 1$. Substituting the four previous formulas in (\ref{Leib1}), the identity (\ref{Leib1}) holds if and only if the coefficients $d_1,d_2,d_3$ and $d_4$ satisfy 
\begin{equation}
\begin{pmatrix}
r^2 & rr' & rr' & r'^2 \\
rs & sr' & rs' & r's' \\
rs & rs' & sr' & r's'\\
s^2 & ss' & ss' & s'^2
\end{pmatrix}\begin{pmatrix}d_1\\d_2\\d_3\\d_4\end{pmatrix}=\begin{pmatrix}r\varepsilon\\0\\0\\-s\varepsilon\end{pmatrix}
\label{sys1}
\end{equation}
whose determinant is equal to $-(rs'-sr')^4\neq 0$. We replace in (\ref{Leib1}) the coefficients $d_\ell$ by their explicit values 
\begin{equation}
d_1=\frac{\varepsilon}{\delta}(rs'^2-sr'^2),~d_2=d_3=\frac{\varepsilon rs}{\delta}(r'-s'),~d_4=\frac{\varepsilon rs}{\delta}(s-r)
\end{equation}
where $\delta=(rs'-sr')^2$. Since $s/r\notin\mathbb{R}$, we can choose $r'=\overline{r}$ and $s'=\overline{s}$ and we get easily the formula (\ref{Leibniz}).
\end{proof}

As an example, we get $-d_1=d_2=d_3=d_4=-\frac{1}{2}i\varepsilon$ for the operator $\Box_{\varepsilon,\mathcal{Q}}=\Box_{\varepsilon}^{[\frac{1-i}{2},\frac{1+i}{2}]}$ chosen in \cite{Cre,CFT}, so that for all piecewise continuous  $f,g:[a,b]\rightarrow\mathbb{R}$ 
\begin{center}
$\displaystyle \Box_{\varepsilon,\mathcal{Q}}(fg)=\Box_{\varepsilon,\mathcal{Q}}(f)g+f\Box_{\varepsilon,\mathcal{Q}}(g)+\frac{1}{2}i\varepsilon[\Box_{\varepsilon,\mathcal{Q}}(f)\Box_{\varepsilon,\mathcal{Q}}(g)-\Box_{\varepsilon,\mathcal{Q}}(f)\boxminus_{\varepsilon,\mathcal{Q}}(g)-\boxminus_{\varepsilon,\mathcal{Q}}(f)\Box_{\varepsilon,\mathcal{Q}}(g)-\boxminus_{\varepsilon,\mathcal{Q}}(f)\boxminus_{\varepsilon,\mathcal{Q}}(g)]$,
\end{center}
where $\boxminus_\varepsilon$ stands for the complex conjugate operator of $\Box_\varepsilon$. (We have corrected here the corresponding formula in \cite{Cre,CFT}.)

\section{Critical points of discrete actions}

According to the previous notation, the discrete actions and lagrangians are related to classical ones as follows.
\begin{definition}
If $\mathcal{L}(t,\mathbf{x}(t),\mathbf{\dot x}(t))$ is a lagrangian depending on $2d+1$ variables then we set $L(t,S(\mathbf{x})(t))=%
\mathcal{L}(t,\mathbf{x}(t),\Box_\varepsilon \mathbf{x}(t))$. Moreover, 
\begin{equation}
\mathcal{A}_{cont}(\mathbf{x})=\int_a^b\mathcal{L}(t,\mathbf{x}(t),\mathbf{\dot x}(t))dt
\label{fpcv}
\end{equation}
and
\begin{equation}
\mathcal{A}_{disc}(\mathbf{x})=\int_a^bL(t,S(\mathbf{x})(t))dt=\int_a^b\mathcal{L}(t,\mathbf{x}(t),\Box_\varepsilon \mathbf{x}(t))dt
\label{fpcv2}
\end{equation}
are the respective classical and discrete actions associated to these lagrangians, defined respectively on $C^1([a,b],\mathbb{C}^d)$ and $\mathcal{C}_{pw}(d,N,\alpha,\beta)$.
\end{definition}
If the terms $x_j(t+\ell\varepsilon )\chi_{-\ell}(t)$ occuring in $S(\mathbf{x})(t)$ (see (\ref{Sx})) are replaced with new variables $\xi _{j,\ell}$, then $L(t,\xi _{1,-N},\ldots ,\xi_{d,N})$ is nothing but $\mathcal{L}(t,\mathbf{x}(t),\Box_\varepsilon \mathbf{x}(t))$ and thus depends on $d(2N+1)+1$
indeterminates.

A fundamental problem is to minimize the action (\ref{fpcv}) under Dirichlet boundary conditions when obviously, every unknown and parameter has to be real. Note that in order to deal with optima instead of critical points of the action (\ref{fpcv2}), we have to handle real valued functions and parameters. 
\begin{theorem}
Let $\mathbf{x}\in \mathcal{C}_{pw}(d,N,\mathbf{\alpha },\mathbf{\beta})$ be a critical point of (\ref{fpcv2}). Then $\mathbf{x}$ satisfies the 
following functional equation
\begin{equation}
\forall j\in \{1,\ldots ,d\},~\sum_{\ell=-N}^N\frac{\partial L}{\partial \xi
_{j,\ell}}(t-\ell\varepsilon ,S(\mathbf{x})(t-\ell\epsilon ))\chi_\ell(t)=0
\label{mcel}
\end{equation}
and the equation (\ref{mcel}) may be returned under an intrinsic form as (\ref{intrinsic}), i.e. 
\begin{equation}
\Box_{-\varepsilon }\frac{\partial \mathcal{L}}{\partial \dot x_j}(t,\mathbf{x}(t),\Box_\varepsilon \mathbf{x}(t))+\frac{\partial \mathcal{L}}{\partial x_j}(t,\mathbf{x}(t),\Box_\varepsilon \mathbf{x}(t))=0,
\label{mcelbis}
\end{equation}
for all $j\in \{1,\ldots ,d\}$.
\end{theorem}

\begin{proof} 
We have for all $\mathbf{h}\in\mathcal{C}_{pw}(d,N,0,0)$
\begin{center}
$\displaystyle \mathcal{A}_{disc}(\mathbf{x}+\eta \mathbf{h})-\mathcal{A}_{disc}(\mathbf{x})=\int_a^b(L(t,S(\mathbf{x}+\eta \mathbf{h})(t))-L(t,S(\mathbf{x})(t)))dt$
$\displaystyle =\eta \int_a^b\sum_{j=1}^d\sum_{\ell=-N}^N\frac{\partial L}{\partial \xi _{j,\ell}}(t,S(\mathbf{x})(t))h_j(t+\ell\varepsilon )\chi
_{-\ell}(t)dt+O(\eta^2)$.
\end{center}
Therefore using Chasles relation and setting $t=\tau-\ell\varepsilon $ in the previous equality, we find the G\^ateaux derivative
\begin{center}
$\begin{array}{rcl}
D\mathcal{A}_{disc}(\mathbf{x})(\mathbf{h}) & = & \displaystyle%
\sum_{j=1}^d\sum_{\ell=-N}^N\int_{a+\ell\varepsilon}^{b+\ell\varepsilon }\frac{%
\partial L}{\partial \xi _{j,\ell}}(\tau -\ell\varepsilon,S(\mathbf{x})(\tau
-\ell\varepsilon ))h_j(\tau )d\tau \\ 
& = & \displaystyle\int_{a-N\varepsilon }^{b+N\varepsilon }\sum_{j=1}^d
h_j(\tau )\sum_{\ell=-N}^N\frac{\partial L}{\partial \xi _{j,\ell}}(\tau
-\ell\varepsilon ,S(\mathbf{x})(\tau-\ell\varepsilon ))\chi _\ell(\tau )d\tau =0,
\end{array}$
\end{center}
which gives (\ref{mcel}) by using Paul Dubois-Reymond lemma in the multidimensional case. By definition of $L(t,S(\mathbf{x})(t))$, we have for all  $\ell\neq 0$,
\begin{center}
$\displaystyle \frac{\partial L}{\partial \xi _{j,\ell}}(t,S(\mathbf{x})(t))=\frac{\partial \mathcal{L}}{\partial \dot x_j}(t,\mathbf{x}(t),\Box_{\varepsilon}\mathbf{x}(t))c_\ell\chi_{-\ell}$
\end{center}
and for $\ell=0$,
\begin{center}
$\displaystyle ~~~~~~\frac{\partial L}{\partial \xi _{j,0}}(t,S(\mathbf{x})(t))= \frac{\partial \mathcal{L}}{\partial \dot x_j}(t,\mathbf{x}(t),\Box_{\varepsilon}\mathbf{x}(t))c_0+ \frac{\partial{\mathcal{L}}}{\partial x_j}(t,\mathbf{x}(t),\Box_{\varepsilon}\mathbf{x}(t))$.
\end{center}
Thus, (\ref{mcel}) is equivalent to
\begin{center}
$\displaystyle\sum_{\ell=-N}^{N}c_\ell\chi_{-\ell}\frac{\partial \mathcal{L}}{\partial \dot x_j} (t-\ell\varepsilon,\mathbf{x}(t-\ell\varepsilon),\Box_{\varepsilon}\mathbf{x}(t-\ell\varepsilon))+\frac{\partial{\mathcal{L}}}{\partial x_j}(t,\mathbf{x}(t),\Box_{\varepsilon}\mathbf{x}(t))=0$.
\end{center}
Since the first term is equal to $\displaystyle \Box_{-\varepsilon }\frac{\partial \mathcal{L}}{\partial \dot x_j}(t,\mathbf{x}(t),\Box_\varepsilon \mathbf{x}(t))$, the result is proved.
\end{proof}

\begin{remark}\rm
When $\Box_\varepsilon$ is not of the shape (\ref{ourboxpq}), we do not have (\ref{Leibniz}) and integration by parts may not be performed. We use instead simple changes of variables. Note also that second order derivatives of the lagrangian occuring in C.E.L. are replaced with time delayed first order ones.
\label{rem21}
\end{remark}

\begin{remark}\rm
In \cite{Cre}, Cresson deals with the case $N=1$ and the operator $\Box_{\varepsilon,\mathcal{Q}}=\Box_\varepsilon^{[\frac{1-i}{2},\frac{1+i}{2}]}$. For any function $f(\mathbf{x},\varepsilon)$ he
defines the $\varepsilon$-dominant part $[f]_\varepsilon $ with a limiting process. He proved that if  
\begin{center}
$\displaystyle \lim_{\varepsilon\to 0}D\mathcal{A}_{disc}(\mathbf{x})=0$,
\end{center} 
then 
\begin{equation}
\displaystyle\left[ \frac{\partial \mathcal{L}}{\partial x_j}-\Box
_{\varepsilon,q}\frac{\partial \mathcal{L}}{\partial \dot x_j}\right]
_\varepsilon =\lim_{\varepsilon\to 0}\left( \frac{\partial \mathcal{L}}{\partial x_j}-\Box
_{\varepsilon,q}\frac{\partial \mathcal{L}}{\partial \dot x_j}\right)=0  
\label{cresson}
\end{equation}
which is an alternative form of (\ref{mcelbis}). Our equation (\ref{mcelbis}) is however different from (\ref{cresson}) at least in three respects. First, the characteristic functions $\chi _{-\ell}(t)$ of the various intervals appear in the sampling process (\ref{ourbox}) as well as in the action (\ref{fpcv2}). The second one is that (\ref{mcelbis}) does not depend on the coefficients nor on the length of formula (\ref{ourbox}). The
last one is the use in (\ref{mcelbis}) of $\Box_{-\varepsilon }$ instead of $-\Box_\varepsilon $. But we have $\displaystyle [\Box_{\varepsilon
,q}f(t)]_\varepsilon=[-\Box_{-\varepsilon ,q}f(t)]_\varepsilon$ for all real valued function $f$ and $t\in [a+\varepsilon ,b-\varepsilon ]$. Indeed, $[\Im(\Box_{\varepsilon,\mathcal{Q}}f)]_\varepsilon =0$ and $\Re(\Box_{-\varepsilon ,q}f)=-\Re(\Box_{\varepsilon,\mathcal{Q}}f)$ since every characteristic function in $\Box_{\varepsilon,\mathcal{Q}}$ is equal to 1.
\label{rem22}
\end{remark}

\section{Quadratic lagrangians in discrete and classical settings}

In this section we deal with a system of $d$ ordinary differential equations of the second order arising from the following lagrangian 
\begin{equation}
\mathcal{L}(t,\mathbf{x},\mathbf{\dot{x}})=\frac 12\tra{\mathbf{\dot x}}P\mathbf{\dot x}+\frac
12\tra{\mathbf{x}}Q\mathbf{x}+\tra{\mathbf{x}}R\mathbf{\dot x}+\tra{J_1}\mathbf{\dot x}+\tra{J_2}\mathbf{x}+J_3,  
\label{regulator}
\end{equation}
where $P(t),Q(t),R(t)\in\mathbb{C}^{d\times d}$, $J_1(t),J_2(t)\in\mathbb{C}^d$ and $J_3(t)$ is a scalar function. Many physical systems might be modelized by such lagrangians, in electromagnetism, quantum mechanics, material science, regulators models and so on.

A convenient setup that we assume from now on is that the coefficients in (\ref{regulator}) are real and smooth, and that for all $t\in [a,b]$, $P(t)$ and $Q(t)$ are symmetric and $R(t)$ is skew-symmetric (the symmetric part of $R(t)$ gives rise to a null lagrangian term in (\ref{regulator})).
\begin{theorem}
Let $\mathcal{L}$ be a quadratic lagrangian such that $\tra{P}=P$, $\tra{Q}=Q
$, $\tra{R}=-R$ and $L$ associated to $\mathcal{L}$ as in (\ref{fpcv2}). The Euler-Lagrange equation associated to (\ref{regulator})
can be written as 
\begin{equation}
-P\ddot {\mathbf{x}}+(-\dot P+2R)\dot {\mathbf{x}}+(\dot R+Q)\mathbf{x}-\dot{J_1}+J_2=0.  \label{elho}
\end{equation}
The equation (\ref{elho}) may be discretized a posteriori to give 
\begin{equation}
-P\Box_\varepsilon (\Box_\varepsilon \mathbf{x})+(-\dot P+2R)\Box
_\varepsilon \mathbf{x}+(\dot R+Q)\mathbf{x}-\dot J_1+J_2=0.
\label{mchomatdiagcom}
\end{equation}
If $\mathbf{x}\in \mathcal{C}_{pw}(d,N,\alpha ,\beta )$ is a critical
point of the action (\ref{fpcv2}), then it must satisfy 
\begin{equation}
\Box_{-\varepsilon }(P\Box_\varepsilon \mathbf{x})-\Box_{-\varepsilon}(R\mathbf{x})+R\Box_\varepsilon\mathbf{x}+Q\mathbf{x}+\Box_{-\varepsilon}J_1+J_2=0.  
\label{mchomat}
\end{equation}
\label{theorem4.2}
\end{theorem}

\begin{proof}
First, (\ref{elho}) is straightforward, since we get 
\begin{equation*}
\frac{\partial \mathcal{L}(t,\mathbf{x},\mathbf{\dot{x}})}{\partial \mathbf{x}}=Q\mathbf{x}+R%
\mathbf{\dot x}+J_2\mbox{ and }\frac{\partial \mathcal{L}(t,\mathbf{x},\mathbf{\dot{x}})}{%
\partial \mathbf{\dot x}}=P\mathbf{\dot x}-R\mathbf{x}+J_1.
\end{equation*}
Next, (\ref{mchomatdiagcom}) is obtained by discretizing the derivatives in (\ref{elho}), i.e. by replacing $\dot{\mathbf{x}},\ddot{\mathbf{x}}$ by $\Box_\varepsilon\mathbf{x},\Box _\varepsilon (\Box _\varepsilon \mathbf{x})$ respectively and the result holds. At last, (\ref{mchomat}) is a consequence of (\ref{mcelbis}). Indeed, using the previous derivatives, (\ref{mcelbis}) gives
\begin{center}
$\Box_{-\varepsilon}(P\Box_\varepsilon\mathbf{x}-R\mathbf{x}+J_1)+Q\mathbf{x}%
+R\Box_\varepsilon\mathbf{x}+J_2=0$,
\end{center}
which ends the proof.
\end{proof}

\begin{remark}\rm
If $P,Q,R$ are independent on time then, for $t\in \mathcal{I}_S$, both equations (\ref{mchomat}) and (\ref{mchomatdiagcom}) are equivalent.
\end{remark}

\begin{remark}\rm
For the so-called linear quadratic control problem, Guibout and Scheeres \cite{GS} give Euler-Lagrange differential equations having a similar structure than (\ref{elho}). 
They suppose that the relation $\dot{\mathbf{x}}(t)=M_1(t)\mathbf{x}(t)+M_2(t)\mathbf{u}(t)$ holds between the dynamical variable $\mathbf{x}$ and the control variable $\mathbf{u}$ and prove that each critical point $(\mathbf{x},\mathbf{p})$ satisfies 
\begin{center}
$\dot {\mathbf{x}}(t)=(M_1-\frac 12M_2P^{-1}\tra{R})\mathbf{x}(t)-M_2P^{-1}\tra{M}_2{\mathbf{p}}(t)$,\\
$\dot {\mathbf{p}}(t)=(\frac 12R\tra{P}^{-1}\tra{M_2}-\tra{M}_1)\mathbf{p}(t)+(\frac 14RP^{-1}\tra{R}-Q)\mathbf{x}(t)$,
\end{center}
where $\mathbf{p}$ stands for a Lagrange multiplier of the constraint.
\end{remark}

Let us investigate now the harmonic oscillator. We have in this case $d=1$, $R=0$, $J_1=J_2=0$ and we set moreover $P(t)=p$, $Q(t)=q$ with $pq<0$.
Solutions of C.E.L. $-p\ddot{x}+qx=0$ are periodic. Let us study the periodicity of solutions of D.E.L. for $\varepsilon$ small enough. If $t$ lies in $\mathcal{I}_S$, (\ref{mchomat2}) may be simplified into 
\begin{eqnarray}
pc_{-1}c_1[\mathbf{x}(t-2\varepsilon)+\mathbf{x}(t+2\varepsilon)]+pc_0(c_{-1}+c_1)[\mathbf{x}(t-\varepsilon)+\mathbf{x}(t+\varepsilon)]
+\nonumber\\
(q+p(c_{-1}^2+c_0^2+c_1^2))\mathbf{x}(t)=0.~~~~~~~~~~~~~~~~~~~~~~~~~~
\label{oscillat} 
\end{eqnarray}
The result below presents an additional characterization of ``suitable" $\Box_\varepsilon$.

\begin{proposition}
Let $\Box_\varepsilon\in\tilde{\mathcal{O}}_{1,\varepsilon}$. The two following properties are equivalent.
\begin{enumerate}
\item[(a)] For all $p$ and $q$ with $pq<0$ and for all $\varepsilon$ small enough, the roots of the characteristic polynomial of (\ref{oscillat}) are of modulus 1.
\item[(b)] For some $k\in\mathbb{R}$, we have 
$\Box_\varepsilon=\Box_\varepsilon^{[\frac{1}{2},\frac{1}{2}]}+ik\Box_\varepsilon^{[1,-1]}$.
\end{enumerate}
\label{thmdroite}
\end{proposition}

\begin{proof}
The characteristic polynomial of the recurrence (\ref{oscillat}) is symmetric that is to say $D(\lambda )=\lambda
^4D(1/\lambda)$. Setting $\mu =\lambda +1/\lambda $, we get a quadratic $E(\mu)$ such that 
\begin{equation*}
E(\mu)=\varepsilon^2 D(\lambda)=\gamma_1\gamma_{-1}\mu ^2+\gamma_0(\gamma_1+\gamma_{-1})\mu
+(\gamma_0^2+\gamma_1^2+\gamma_{-1}^2-2\gamma_1\gamma_{-1}+\frac qp\varepsilon^2).
\end{equation*}
$(a)\Rightarrow(b)$. We look for the parameters $\gamma _{-1},\gamma _0,\gamma _1$ for which $D(\lambda)$ has roots on the unit circle, for $\varepsilon$ small enough and for all $p,q$, $pq<0$. This amounts to say that $E(\mu )$ has only real roots in $\left[-2,2\right]$ for all $p,q$, $pq<0$, provided $\varepsilon$ is small enough. Let us generalize a little bit. Let $\alpha ,\beta ,\gamma $ be three complex numbers, then the solutions of the quadratic $\alpha y^2+\beta y+\gamma +\varepsilon=0$ are in $\left[-2,2\right]\subset \mathbb{R}$ for all $\varepsilon$ small enough if
and only if we have $\alpha ,\beta ,\gamma \in \mathbb{R}$ and 
\[
|\beta |\leq 4|\alpha |,~~|\gamma |\leq 4|\alpha |,~~4\alpha \gamma\leq \beta ^2,~~8|\beta |\leq 16|\alpha |+4\gamma sgn(\alpha),  
\]
as shows explicit computations using the usual solution of a quadratic. Since $\Box_\varepsilon\in\tilde{\mathcal{O}}_{1,\varepsilon}$, we have $\Box_\varepsilon 1=\gamma_{-1}+\gamma_0+\gamma_1=0$ and $\Box_\varepsilon t=\gamma_1-\gamma_{-1}=1$ inside $\mathcal{I}_S$ if and only if $\gamma_1=r$, $\gamma_{-1}=r-1$ and $\gamma_0=-2r+1$ for some $r\in\mathbb{C}$. So, $\alpha=r(r-1)$, $\beta=(2r-1)^2$ and $\gamma=4r^2-4r+2$ are real. Therefore we obtain $\Re(r)=1/2$ and easy computations lead to $(b)$.\\
$(b)\Rightarrow(a)$. If $(b)$ holds then setting $r=\frac{1}{2}+ik$, we get $\gamma_1=r$, $\gamma_{-1}=r-1$ and $\gamma_0=-2r+1$. Direct computations show that the solutions of
\begin{equation}
E(\mu)=r(r-1)\mu^2-(2r-1)^2\mu+(4r^2-4r+2)-\omega^2\varepsilon^2,
\label{quadric}
\end{equation} 
where $\omega^2=-q/p$, are in $[-2,2]\subset\mathbb{R}$ for all $\varepsilon\leq 1/(|\omega|\sqrt{1+4k^2})$. 
\end{proof}

\begin{remark}\rm
We recover $\Box_{\varepsilon,\mathcal{Q}}$ and $\Box_{\varepsilon,\mathcal{S}}$ by setting $k=-1/2$ and $k=0$ in $(b)$ respectively.
\end{remark}

\begin{remark}\rm
If $\Box_\varepsilon$ lies in $\mathcal{O}_{1,\varepsilon}$, the property $(a)$ is equivalent to the following inequalities for $\gamma_{-1},\gamma_0,\gamma_1\in\mathbb{C}$
\begin{center}
$4|\gamma _1\gamma _{-1}|-|\gamma _0(\gamma _1+\gamma _{-1})|\geq 0$,~~$4|\gamma
_1\gamma _{-1}|-|\gamma _0^2+\gamma _1^2+\gamma _{-1}^2-2\gamma _1\gamma
_{-1}|\geq 0$,
$\gamma _0^2(\gamma _1+\gamma _{-1})^2-4\gamma _1\gamma _{-1}(\gamma
_0^2+\gamma _1^2+\gamma _{-1}^2-2\gamma _1\gamma _{-1})\geq 0$ and
$16|\gamma _1\gamma _{-1}|+4\gamma _0^2+4\gamma _1^2+4\gamma _{-1}^2-8\gamma
_1\gamma _{-1}-8|\gamma _0(\gamma _1+\gamma _{-1})|\geq 0$.
\end{center}
So there exists operators $\Box_\varepsilon\notin\tilde{\mathcal{O}}_{1,\varepsilon}$ satisfying $(a)$ but not $(b)$.
\end{remark}

\section{Convergence of functional equations D.E.L. to differential equations C.E.L.}

In this section, we address the problem of convergence of D.E.L. to C.E.L. in the case of quadratic lagrangians. Roughly, this convergence
property characterizes $\tilde{\mathcal{O}}_{N,\varepsilon}$. In the remainder, the matrices $P,Q,R$ in (\ref{regulator}) are dependent on time, $P,Q$ are symmetric and $R$ is skew-symmetric. To begin with, let us give the following

\begin{lemma}
The left hand side of D.E.L. (\ref{mchomat}) is equal to
\begin{equation*}
\Theta(\mathbf{x})(t)=\sum_{\tiny
\begin{array}[t]{c}
-2N\leq \ell\leq 2N \\ 
-N\leq j\leq N \\ 
|\ell+j|\leq N
\end{array}
}\hspace{-0.2cm}c_{\ell+j}c_j\chi _j(t)\chi _{-\ell}(t)P(t-j\varepsilon )\mathbf{x}(t+\ell\varepsilon )+Q(t)\mathbf{x}(t)+
\end{equation*}
\vspace{-0.8cm}
\begin{equation}
\hspace{1.7cm}\sum_{\ell=-N}^N\hspace{-0.2cm}\chi_{-\ell}(t)(c_\ell R(t)-c_{-\ell}R(t+\ell\varepsilon ))\mathbf{x}(t+\ell\varepsilon )+\Box_{-\varepsilon} J_1(t)+J_2(t)
\label{mchomat2}
\end{equation}
for all $t\in [a,b]$. 
\end{lemma}

\begin{proof}
The proof is straightforward by using several times the formula (\ref{ourbox}) for $\Box_\varepsilon$ and $\Box_{-\varepsilon}$, and showing that (\ref{mchomat}) may be returned as $\Theta(\mathbf{x})(t)=0$ for all $t\in[a,b]$.
\end{proof}

\begin{definition}
We say that D.E.L. (\ref{mchomat}) converges to C.E.L. (\ref{elho}) as $\varepsilon $ tends to 0 if, for all quadratic lagrangian $\mathcal{L}$ as in (\ref{regulator}) and all $\mathbf{x}\in\mathcal{C}^2([a,b],\mathbb{R}^d)$, 
$$\displaystyle \lim_{\varepsilon\to 0}\Theta(\mathbf{x})(t)=-P\ddot {\mathbf{x}}+(-\dot P+2R)\dot {\mathbf{x}}+(\dot R+Q)\mathbf{x}-\dot
{J_1}+J_2$$
locally uniformly in $]a,b[$, in the norm of $\mathcal{L}^\infty ([a+\delta ,b-\delta ])$ for all $\delta >0$.
\end{definition}
The remainder of this section is devoted to the proof of the main result below.
\begin{theorem}
Let $\Box_\varepsilon\in\mathcal{O}_{N,\varepsilon}$. The following six properties are equivalent.
\begin{enumerate}
\item[(a)] D.E.L. (\ref{mchomat}) converges to C.E.L. (\ref{elho}) as $\varepsilon $ tends to 0.
\item[(b)] For all $\mathbf{x}\in\mathcal{C}^2([a,b],\mathbb{R}^d)$, $\displaystyle \lim_{\varepsilon\to 0}\Box_\varepsilon\mathbf{x}(t)=\dot{\mathbf{x}}(t)$ locally uniformly in $]a,b[$.
\item[(c)] For all $\mathbf{x}\in\mathcal{C}^2([a,b],\mathbb{R}^d)$, $\displaystyle\lim_{\varepsilon\to 0}\Box_{-\varepsilon}\mathbf{x}(t)=-\dot{\mathbf{x}}(t)$ locally uniformly in $]a,b[$.
\item[(d)] The functions $t\rightarrow \Box_\varepsilon 1$ and $t\rightarrow\Box_\varepsilon t$ converge respectively to 0 and 1 locally uniformly in $]a,b[$.
\item[(e)] $\Box_\varepsilon\in\tilde{\mathcal{O}}_{N,\varepsilon}$.
\item[(f)] There exists complex numbers $k_1,\ldots ,k_{2N-1}$ such that 
\begin{equation}
\Box_\varepsilon\mathbf{x}(t)=\Box_\varepsilon^{[1,0]}\mathbf{x}(t)+\hspace{-0.3cm}\sum_{\ell=-(N-1)}^{N-1}k_{\ell+N}%
\Box_\varepsilon^{[1,-1]}\mathbf{x}(t-\ell\varepsilon).
\label{diam2}
\end{equation}
\end{enumerate}
\label{Theorem-conv-schemas}
\end{theorem}

\begin{proof}
We will prove that $(a)\Rightarrow (c)\Rightarrow (d)\Leftrightarrow (e)\Leftrightarrow (f)$, $(d)\Rightarrow (b)\Rightarrow (c)\Rightarrow (a)$.\\
$(a)\Rightarrow (c)$. By assumption, the function $\Theta(0)=\Box_{-\varepsilon} J_1+J_2$ must tend to the constant term in (\ref{elho}) that is $-\dot{J}_1+J_2$, for all $J_1$ and $J_2$, and the result holds.\\
$(c)\Rightarrow (d)$. We choose $\mathbf{x}$ as a linear function $t\rightarrow\mathbf{x_1}t+\mathbf{x_0}$ which checks the property in $(c)$. As  $\Box_{-\varepsilon}$ acts component-wise, $\Box_{-\varepsilon}1$ and $\Box_{-\varepsilon}t$ tend to 0 as $\varepsilon$ tends to 0. But since 
\begin{equation}
\Box_\varepsilon 1=\frac{1}{\varepsilon}\sum_{\ell=-N}^N\gamma_\ell\chi_{-\ell}(t)\mbox{ and }\Box_\varepsilon t=t\Box_\varepsilon 1+\frac{1}{2}\sum_{\ell=-N}^N\ell(\gamma_\ell\chi_{-\ell}(t)-\gamma_{-\ell}\chi_\ell(t)),
\label{cdonned}
\end{equation}
for all $\varepsilon\neq 0$, we have inside $\mathcal{I}_S$ the relations $\Box_\varepsilon 1=-\Box_{-\varepsilon}1$ and $\Box_\varepsilon t=-\Box_{-\varepsilon}t+t(\Box_\varepsilon 1+\Box_{-\varepsilon}1)$. This shows that $(d)$ holds.\\ 
$(d)\Leftrightarrow (e)$. 
For all $\delta>0$ and $\varepsilon$ small enough, if $t\in[a+\delta,b-\delta]$ then $t\in\mathcal{I}_S$, and we see that the assumptions in $(d)$ are equivalent to (\ref{cdonned}) and in turn to the two linear equations
\begin{equation}
\sum_{\ell=-N}^N\gamma_\ell=0\mbox{ and }\frac{1}{2}\sum_{\ell=-N}^N\ell(\gamma_\ell-\gamma_{-\ell})=1.
\label{conditions}
\end{equation}
If $\Box_\varepsilon\in\mathcal{O}_{N,\varepsilon}$ satisfies these two equations then $\Box_\varepsilon 1\rightarrow 0$ and $\Box_\varepsilon t\rightarrow 1$ uniformly locally in $]a,b[$, that is $\Box_\varepsilon$ satisfies the statements in $(d)$ and conversely. Thus, we have proved that $\Box_\varepsilon\in\tilde{\mathcal{O}}_{N,\varepsilon}$ if and only if $(d)$ holds.\\
$(e)\Leftrightarrow (f)$. The two equations (\ref{conditions}) being independent, $\rm codim\it(\tilde{\mathcal{O}}_{N,\varepsilon})=2$. If $\hat{\mathcal{O}}_{N,\varepsilon}$ is the set of operators defined by (\ref{diam2}), then $\rm codim\it(\hat{\mathcal{O}}_{N,\varepsilon})=2$. We easily see that each operator $\Box_\varepsilon\in\hat{\mathcal{O}}_{N,\varepsilon}$ checks the two equations in $(d)$, that is $\hat{\mathcal{O}}_{N,\varepsilon}\subset \tilde{\mathcal{O}}_{N,\varepsilon}$ and lastly, $\hat{\mathcal{O}}_{N,\varepsilon}=\tilde{\mathcal{O}}_{N,\varepsilon}$.\\
$(d)\Rightarrow (b)$. We use Taylor-Lagrange formula to obtain an expansion of $\Box_\varepsilon\mathbf{x}$ as
$$\Box_\varepsilon\mathbf{x}(t)=\sum_{\ell=-N}^Nc_\ell\chi_{-\ell}(t)\left[\mathbf{x}(t)+\ell\varepsilon\dot{\mathbf{x}}(t)+\int_t^{t+\ell\varepsilon}(t+\ell\varepsilon-s)\ddot{\mathbf{x}}(s)ds\right]$$
\vspace{-0.4cm}
$$=(\Box_\varepsilon 1)\mathbf{x}(t)+(\Box_\varepsilon t-t\Box_\varepsilon 1)\dot{\mathbf{x}}(t)+\int_a^bG(s,t)\ddot{\mathbf{x}}(s)ds.$$
By using $\tilde\chi_{[0,\ell\varepsilon]}=\chi_{[0,\ell\varepsilon]}$ if $\ell\geq 0$ and $\tilde\chi_{[0,\ell\varepsilon]}=\chi_{[\ell\varepsilon,0]}$ if $\ell\leq 0$, the previous kernel $G(s,t)$ is equal to 
\begin{equation}
G(s,t)=\chi_0(s)\chi_0(t)\sum_{\ell=-N}^N(t+\ell\varepsilon-s)c_\ell\chi_{-\ell}(t)\tilde\chi_{[0,\ell\varepsilon]}(s-t).
\label{noyau}
\end{equation}
Let $\delta>0$, $V_1=\mathcal{L}^\infty([a,b])$, $V_2=\mathcal{L}^\infty([a+\delta,b-\delta])$ and $V_3=\mathcal{L}^\infty([a,b]\times[a+\delta,b-\delta])$. Then, for all function $\mathbf{x}\in\mathcal{C}^2([a,b])$, if $\varepsilon<\delta/N$, the norm $\|\Box_\varepsilon\mathbf{x}-\dot{\mathbf{x}}\|_{V_2}$ is bounded by 
$$\|\Box_\varepsilon 1\|_{V_2}\left(\|\mathbf{x}\|_{V_1}+\max(|a|,|b|)\|\dot{\mathbf{x}}\|_{V_1}\right)+\|\Box_\varepsilon t-1\|_{V_2}\|\dot{\mathbf{x}}\|_{V_1}+(b-a)\|G\|_{V_3}\|\ddot{\mathbf{x}}\|_{V_1}.$$
The first and the second terms converge to 0 as $\varepsilon$ tends to 0 help to $(d)$. In order to prove that the last term tends also to 0 as $\varepsilon$ tends to 0, we note that 
$$G(s,t)=(t-s)\Box_\varepsilon 1+\sum_{\ell=-N}^N(\ell\varepsilon \tilde\chi_{[0,\ell\varepsilon]}(s-t)+(t-s)(\tilde\chi_{[0,\ell\varepsilon]}(s-t)-1))c_\ell\chi_{-\ell}(t).$$
Obviously, first and second terms tend to 0 as $\varepsilon$ tends to 0. To deal with the third term, we distinguish two cases. If $t\leq s\leq t+k\varepsilon$ for some $k\in[-N,N]$, the term is bounded by $N\varepsilon\sum_{\ell=-N}^N|c_\ell|$ while if $|s-t|>N\varepsilon$, it is bounded by $|s-t|.||\Box_\varepsilon 1||_{V_2}$, and thus, by the sum of these two upper bounds. Collecting altogether, we find 
$$\|G\|_{V_3}\leq 2\max(|b-a|,2|a|,2|b|)\|\Box_\varepsilon 1\|_{V_2}+N\varepsilon\sum_{\ell=-N}^N|c_\ell|+\varepsilon \sum_{\ell=-N}^N|\ell c_\ell|.$$ 
Therefore, $\|G\|_{V_3}$ tends to 0 and thus, also $\|\Box_\varepsilon\mathbf{x}-\dot{\mathbf{x}}\|_{V_1}$ as $\varepsilon$ tends to 0 and the result holds.\\
$(b)\Rightarrow (c)$. The two operators $\Box_\varepsilon$ and $\Box_{-\varepsilon}$ are linked by the following property. If $\mathbf{x}_1\in\mathcal{C}_{pw}(d,N,\alpha,\beta)$ then the function $\mathbf{x}_2$ defined by $\mathbf{x}_2(t)=\mathbf{x}_1(a+b-t)$ is in $\mathcal{C}_{pw}(d,N,\beta,\alpha)$ and satisfies
$$(\Box_{-\varepsilon} \mathbf{x}_1)(t)=-(\Box_\varepsilon\mathbf{x}_2)(a+b-t).$$
Indeed, this formula comes from (\ref{ourbox}) and the fact that $\chi_{-j}(s)=\chi_j(a+b-s)$ for all $j$. Since $\dot{\mathbf{x}_2}(t)=-\dot{\mathbf{x}_1}(a+b-t)$, the convergence of $\Box_\varepsilon\mathbf{x}_2$ to $\dot{\mathbf{x}}_2$ for all $\mathbf{x}_2\in\mathcal{C}_{pw}(d,N,\beta,\alpha)$ implies the convergence of $\Box_{-\varepsilon}\mathbf{x}_1$ to $-\dot{\mathbf{x}}_1$ for all $\mathbf{x}_1\in\mathcal{C}_{pw}(d,N,\alpha,\beta)$.\\
$(c)\Rightarrow (a)$. We use again the formula (\ref{mchomat2}) and we deal with its non-constant terms. When we develop $\mathbf{x}(t+\ell\varepsilon)$ help to the Taylor-Mac Laurin formula of order $2$, we get 
\begin{equation}
\Theta(\mathbf{x})(t)=-P_\varepsilon (t)\ddot {\mathbf{x}}(t)+2R_\varepsilon (t)\dot {\mathbf{x}%
}(t)+Q_\varepsilon (t)\mathbf{x}(t)+\Box _{-\varepsilon
}J_1(t)+J_2(t)+\tilde{\Theta}(\mathbf{x})(t),  \label{mchomat33}
\end{equation}
where notation is explained hereafter. First, let us note that $\tilde{\Theta}(\mathbf{x})(t)$ is a combination of two kinds of terms  $M(\ddot{\mathbf{x}}(t+\theta_\ell\ell\varepsilon)-\ddot{\mathbf{x}}(t))$, for some $\theta_\ell\in]0,1[$ and $\ell\in\{-2N,\ldots,2N\}$, $M$ being either equal to $P(t-j\varepsilon)$ or $R(t+\ell\varepsilon)$. When $t\in\mathcal{I}_S$, we see that those terms tend to 0 as $\varepsilon$ tends to 0. Next, some lengthy computations show that 
\vspace{0.3cm}
\begin{center}
$\begin{array}{lll}
P_\varepsilon & = & -\frac{1}{2}P[\Box_{-\varepsilon}\Box_\varepsilon t^2-2t\Box_{-\varepsilon}\Box_\varepsilon t+t^2\Box_{-\varepsilon}\Box_\varepsilon 1]+\mathcal{S}_P,\vspace{0.3cm}\\
Q_\varepsilon & = & Q+P\Box_{-\varepsilon}\Box_\varepsilon 1+R(\Box_\varepsilon 1+\Box_{-\varepsilon} 1)-\dot{P}[(\Box_\varepsilon 1)(\Box_\varepsilon t-t\Box_\varepsilon 1)]+\vspace{0.2cm}\\
 & & \ddot{P}[(\Box_\varepsilon 1)(\Box_\varepsilon t^2-2t\Box_\varepsilon t+t^2\Box_\varepsilon 1)]+\dot{R}(-\Box_{-\varepsilon}t+t\Box_{-\varepsilon}1)+\mathcal{S}_Q,\vspace{0.3cm}\\
R_\varepsilon & = & \frac{1}{2}P(\Box_{-\varepsilon}\Box_\varepsilon t-t\Box_{-\varepsilon}\Box_\varepsilon 1)+\frac{1}{2}R(\Box_\varepsilon t-\Box_{-\varepsilon}t+t\Box_{-\varepsilon} 1-t\Box_\varepsilon 1)\vspace{0.2cm}\\
& & -\frac{1}{2}\dot{P}[(\Box_\varepsilon t-t\Box_\varepsilon 1)^2-(\Box_\varepsilon 1)(\Box_\varepsilon t^2-2t\Box_\varepsilon t+t^2\Box_\varepsilon 1)]+\mathcal{S}_R.
\end{array}$
\end{center}
\vspace{0.2cm}
$\mathcal{S}_P$, $\mathcal{S}_Q$ and $\mathcal{S}_R$ are three explicit matricial combinations of $P,\dot{P},\ddot{P},R,\dot{R},\ddot{R}$ which may be roughly upper bounded as follows
\begin{center}
$\displaystyle \|\mathcal{S}_l\|_{\mathcal{L}^\infty([a,b])}\leq \gamma^2\bar{N}(\varepsilon+\chi_{\overline{\mathcal{I}_S}})\left(\sum_{\ell=0}^3\|P^{(\ell)}\|_{\mathcal{L}^\infty([a,b])}+\sum_{\ell=0}^3\|R^{(\ell)}\|_{\mathcal{L}^\infty([a,b])}\right)$.
\end{center}
Here, $\gamma=\max(1,\max_\ell|\gamma_\ell|)$ and $\bar{N}=(2N+1)(4N+1)N^5$ are two fixed numbers w.r.t. $\varepsilon$. As a consequence, for all small  $\delta>0$, $\chi_{\overline{\mathcal{I}_S}}(t)$ and accordingly $\mathcal{S}_l(t)$ tend to 0 uniformly in $[a+\delta,b-\delta]$ as $\varepsilon$ tends to 0. Finally, since $(c)$ holds, inspection of each coefficient in the previous formulas shows that $P_\varepsilon(t)$, $Q_\varepsilon(t)$, $R_\varepsilon(t)$ and $\tilde{\Theta}(\mathbf{x})(t)$ tend respectively to $P(t)$, $Q(t)+\dot{R}(t)$, $R(t)-\frac{1}{2}\dot{P}(t)$ and $0$ as $\varepsilon$ tends to $0$. The convergence of D.E.L. is thus ensured for all lagrangian $\mathcal{L}$, for all function $\mathbf{x}$ and for all $t\in[a+\delta,b-\delta]$, which ends the proof.
\end{proof}

\begin{remark}\rm
We note that inside the safety interval $\mathcal{I}_S$, for each operator $\Box_\varepsilon\in\mathcal{O}_{N,\varepsilon}$, the three formulas hold:
$\displaystyle\Box_{-\varepsilon}\Box_\varepsilon 1=(\Box_\varepsilon 1)^2$, $\Box_{-\varepsilon}\Box_\varepsilon t=t (\Box_\varepsilon 1)^2$ and 
$\displaystyle\Box_{-\varepsilon}\Box_\varepsilon t^2=t^2(\Box_\varepsilon 1)^2+2(\Box_\varepsilon 1)\left(\sum_{\ell=-N}^N\ell^2c_\ell\right)-2(\Box_\varepsilon t)^2$. More generally, there exists polynomial formulas for iterates of $\Box_\varepsilon$ acting on the polynomials as expressions of $\Box_\varepsilon t^k$ for $k\in\mathbb{N}$. 
\end{remark}

\begin{remark}\rm
As mentioned in Remark \ref{rem22}, Cresson defined extremal curves of the action as functions $\mathbf{x}(t)$ such that $[D\mathcal{A}_{disc}(\mathbf{x})]_\varepsilon=0$, that is to say the dominant part of the Fr\'echet derivative of the action is zero. Obviously, this does not ensure that $\mathcal{A}_{disc}$ is extremal at $\mathbf{x}$. The matter of the convergence of the Euler-Lagrange operator restricted to extremal curves of H\"{o}lderian regularity $C^\beta([a-\varepsilon,b+\varepsilon],\mathbb{R}^d)$ for some $\beta>0$ forms part of the characterization of Discrete Euler Lagrange equations. In contrast, we work here with all curves $\mathbf{x}\in C^2([a,b],\mathbb{R}^d)$, not necessarily extremal.
\end{remark}

\begin{remark}\rm
In the terminology of $\Gamma$-convergence, see for instance \cite{SLO}, the preceding proof implies that for all $t_0\in ]a,b[$, if we denote $ev_{t_0}$ the evaluation map at $t_0$, then the composite mapping $\mathbf{x}\rightarrow ev_{t_0}\circ\Theta(\mathbf{x})$ is $\Gamma$-convergent to 
\begin{center}
$\mathbf{x}\rightarrow ev_{t_0}\circ(-P\ddot {\mathbf{x}}+(-\dot P+2R)\dot {\mathbf{x}}+(\dot R+Q)\mathbf{x}-\dot{J_1}+J_2)$.
\end{center}
Note that the authors in \cite{SLO} proved the $\Gamma$-convergence of actions which are quadratic w.r.t. $\dot{\mathbf{x}}$ and defined on some spaces of piecewise affine maps.
\end{remark}


\end{document}